\documentclass[12pt]{article}

\usepackage{amsmath,amsthm,amssymb}
\usepackage[top=30truemm,bottom=30truemm,left=25truemm,right=25truemm]{geometry}
\usepackage[dvips]{graphicx,color,psfrag}

\theoremstyle{plain}

\newtheorem{definition}{Definition}
\newtheorem{thm}[definition]{Theorem}

\newtheorem{cor}[definition]{Corollary}

\DeclareMathOperator{\Ima}{Im}

\title{Remark on Laquer's theorem for circulant determinants}
\author{Naoya Yamaguchi and Yuka Yamaguchi}

\begin{document}

\maketitle

\begin{abstract}
Olga Taussky-Todd suggested the problem of determining the possible values of integer circulant determinants. 
To solve a special case of the problem, Laquer gave a factorization of circulant determinants. 
In this paper, 
we give a modest generalization of Laquer's theorem. 
Also, we give an application of the generalization to integer group determinants. 
\end{abstract}

\section{Introduction}

For a finite group $G$, 
let $x_{g}$ be an indeterminate for each $g \in G$, and let $\mathbb{Z}[x_{g}]$ be the multivariate polynomial ring in $x_{g}$ over $\mathbb{Z}$. 
The group determinant $\Theta(G)$ of $G$ is defined as follows (\cite{Frobenius1968gruppen}, \cite{Frobenius1968gruppencharaktere}; see also \cite{MR1659232}, \cite{MR1085397}, \cite{johnson2019group}): 
$$
\Theta(G) := \det{\left( x_{g h^{- 1}} \right)}_{g, h \in G} \in \mathbb{Z}[x_{g}]. 
$$
When $G$ is a cyclic group $\mathbb{Z} / n \mathbb{Z} = \left\{ \overline{0}, \overline{1}, \ldots, \overline{n - 1} \right\}$, 
the polynomial $\Theta(G)$ is called a circulant determinant and written as $C_{n}(x_{1}, x_{2}, \ldots, x_{n})$, 
where $x_{i + 1}$ denotes $x_{\overline{i}}$. 
When $x_{1}, x_{2}, \ldots, x_{n}$ are integers, 
$C_{n}(x_{1}, x_{2}, \ldots, x_{n})$ is called an integer circulant determinant. 
Olga Taussky-Todd suggested the problem of determining the possible values of integer circulant determinants \cite{MR550657}; 
that is, determining the set 
$$
S(n) := \{ C_{n}(x_{1}, x_{2}, \ldots, x_{n}) \mid x_{1}, x_{2}, \ldots, x_{n} \in \mathbb{Z} \}. 
$$ 
To solve the $n = 2 p$ case, where $p$ is an odd prime, 
Laquer \cite[Theorem~2]{MR624127} gave the following theorem: 
{\it 
Let $n = r s$, where $r$ and $s$ are relatively prime. 
Then 
\begin{align*}
C_{n}(x_{1}, x_{2}, \ldots, x_{n}) = \prod_{i = 0}^{s - 1} C_{r} \left( y_{1}^{i}, y_{2}^{i}, \ldots, y_{r}^{i} \right), \quad y_{j}^{i} := \sum_{k = 0}^{s - 1} \zeta_{s}^{i (k r + j - 1)} x_{k r + j}, 
\end{align*}
where $\zeta_{s}$ is a primitive $s$-th root of unity.} 
We call this theorem Laquer's theorem. 
From this theorem, 
when $r$ is odd and $s = 2$, 
an integer circulant determinant $C_{n}(x_{1}, x_{2}, \ldots, x_{n})$ can be written as a product of $2$ integer circulant determinants. 
As a result of this consideration, 
Laquer \cite[Theorem~10]{MR624127} showed that 
$$
S(2 p) = (\mathbb{Z}_{2}^{*} \cup 4 \mathbb{Z}) \cap (\mathbb{Z}_{p}^{*} \cup p^{2} \mathbb{Z}), 
$$
where $\mathbb{Z}_{m}^{*} := \{ a \in \mathbb{Z} \mid \gcd(a, m) = 1\}$ for any positive integer $m$.

Laquer's theorem was proved by using a special case of Dedekind's theorem. 
For a finite group $G$, 
let $\widehat{G}$ be a complete set of representatives of the equivalence classes of irreducible representations of $G$ over $\mathbb{C}$. 
Dedekind's theorem is as follows (e.g., \cite{MR1554141}, \cite{MR1606416}, \cite{MR803326},  \cite{MR3622295}): 
{\it When $G$ is abelian, $\Theta(G)$ can be factorized into irreducible polynomials over $\mathbb{C}$ as 
$$
\Theta(G) = \prod_{\chi \in \widehat{G}} \sum_{g \in G} \chi(g) x_{g}. 
$$
}

Let $\Theta(G) \left[ x_{g} \mapsto y_{g} \right]$ be the polynomial obtained by replacing the variable $x_{g}$ in $\Theta(G)$ with $y_{g}$ for any $g \in G$. 
From Dedekind's theorem, we obtain a modest generalization of Laquer's theorem. 

\begin{thm}\label{thm:1}
Let $G = H \times K$ be a direct product of finite abelian groups. 
Then we have 
$$
\Theta(G) = \prod_{\chi \in \widehat{K}} \Theta(H) \left[ x_{h} \mapsto \sum_{k \in K} \chi(k) x_{h k} \right]. 
$$
\end{thm}

We apply this theorem to integer group determinants. 
A group determinant called an integer group determinant when its variables are integers. 
For a finite group $G$, let 
\begin{align*}
S(G) := \{ \det{( x_{g h^{-1}} )_{g, h \in G}} \mid x_{g} \in \mathbb{Z} \}, \quad 
S(G)_{{\rm even}} := S(G) \cap 2 \mathbb{Z}. 
\end{align*}

\begin{thm}\label{thm:2}
Let $G = H \times ( \mathbb{Z} / 2 \mathbb{Z} )^{l}$, where $H$ is a finite abelian group, 
and let 
$$
M := \max{ \{ m \in \mathbb{N} \mid 2^{m} \: \text{divides every element of} \: \: S(H)_{{\rm even}} \} }. 
$$
Then we have 
$$
S(G)_{{\rm even}} \subset 2^{M \cdot 2^{l}} \mathbb{Z}. 
$$
\end{thm}

It is well-known that  $S( \mathbb{Z} / 2 \mathbb{Z} )_{{\rm even}} = 4 \mathbb{Z}$. 
Also, Kaiblinger \cite[Theorem~1.1, Example~3.3]{MR2914452} showed that $S(\mathbb{Z} / 2^{n} \mathbb{Z})_{{\rm even}} \subset 2^{n + 2} \mathbb{Z}$ and $S(\mathbb{Z} / 2^{n} \mathbb{Z})_{{\rm even}} \not\subset 2^{n + 3} \mathbb{Z}$ for $n \geq 2$. 
Therefore, as a corollary of Theorem~$\ref{thm:2}$, 
we have the following. 

\begin{cor}\label{cor:3}
For any $n \geq 2$, the following hold: 
\begin{enumerate}
\item[$(1)$] For $G = ( \mathbb{Z} / 2 \mathbb{Z} )^{n} = \mathbb{Z} / 2 \mathbb{Z} \times ( \mathbb{Z} / 2 \mathbb{Z} )^{n - 1}$, 
we have $S(G)_{{\rm even}} \subset 2^{2 \cdot 2^{n - 1}} \mathbb{Z} = 2^{2^{n}} \mathbb{Z}$; 
\item[$(2)$] For $G = \mathbb{Z} / 2^{n} \mathbb{Z} \times \mathbb{Z} / 2 \mathbb{Z}$, 
we have $S(G)_{{\rm even}} \subset 2^{2 (n + 2)} \mathbb{Z}$. 
\end{enumerate}
\end{cor}

For the $n = 2$ and $3$ cases of $G$ in $(1)$, 
$S(G)$ were determined in \cite{MR3879399} and \cite{MR4056860}, respectively, 
and for the $n = 2$ case of $G$ in $(2)$, 
$S(G)$ was determined in \cite{MR4056860}: 
\begin{align*}
S( ( \mathbb{Z} / 2 \mathbb{Z})^{2} ) &= \{ 4 m + 1, \: 2^{4} (2 m + 1), \: 2^{6} m \mid m \in \mathbb{Z} \}, \\ 
S( ( \mathbb{Z} / 2 \mathbb{Z})^{3} ) &= \{ 8 m + 1, \: 2^{8} (4 m + 1), \: 2^{12} m \mid m \in \mathbb{Z} \}, \\ 
S( \mathbb{Z} / 4 \mathbb{Z} \times \mathbb{Z} / 2 \mathbb{Z} ) &= \{ 8 m + 1, \: 2^{8} m \mid m \in \mathbb{Z} \}. 
\end{align*}
These results imply that the upper inclusions in Corollary~$\ref{cor:3}$~$(1)$ and $(2)$ are best possible in the sense that 
\begin{align*}
S((\mathbb{Z} / 2 \mathbb{Z} )^{n})_{{\rm even}} \not\subset 2^{2^{n} + 1} \mathbb{Z}, \quad 
S(\mathbb{Z} / 4 \mathbb{Z} \times \mathbb{Z} / 2 \mathbb{Z})_{{\rm even}} \not\subset 2^{9} \mathbb{Z}, 
\end{align*}
where $n = 2, 3$. 
Furthermore, 
for $G = (\mathbb{Z} / 2 \mathbb{Z} )^{n}$ with an arbitrary $n$, 
putting $x_{e} = 2$ and $x_{g} = 0$ for all $g \neq e$, where $e$ is the unit element of $G$, 
we have $\Theta(G) = 2^{2^{n}}$. 
That is, 
$
S(G)_{{\rm even}} \not\subset 2^{2^{n} + 1} \mathbb{Z}. 
$

\section{Proof of Theorems}

\begin{proof}[Proof of Theorem~$\ref{thm:1}$]
From Dedekind's theorem and $\widehat{G} = \widehat{H} \times \widehat{K}$, we have
\begin{align*}
\Theta(G) 
&= \prod_{\chi \in \widehat{G}} \sum_{g \in G} \chi(g) x_{g} \\ 
&= \prod_{\chi \in \widehat{H}} \prod_{\chi' \in \widehat{K}} \sum_{h \in H} \sum_{k \in K} \chi(h) \chi'(k) x_{h k} \\ 
&= \prod_{\chi' \in \widehat{K}} \prod_{\chi \in \widehat{H}} \sum_{h \in H} \chi(h) \left( \sum_{k \in K} \chi'(k) x_{h k} \right) \\ 
&= \prod_{\chi' \in \widehat{K}} \Theta(H) \left[ x_{h} \mapsto \sum_{k \in K} \chi'(k) x_{h k} \right]. 
\end{align*}
\end{proof}

We show that Laquer's theorem is a special case of Theorem~$\ref{thm:1}$. 
Let $G = \mathbb{Z} / n \mathbb{Z}$ and $n = r s$ where $r$ and $s$ are relatively prime. 
Since 
$$
G = \langle s \rangle \times \langle r \rangle = \left\{ \overline{a s} + \overline{b r} \mid 0 \leq a \leq r - 1, \: 0 \leq b \leq s - 1 \right\}, 
$$
from Theorem~$\ref{thm:1}$, we have 
\begin{align*}
C_{n}(x_{\overline{0}}, x_{\overline{1}}, \ldots, x_{\overline{n - 1}}) = \prod_{i = 0}^{s - 1} \prod_{l = 0}^{r - 1} \sum_{a = 0}^{r - 1} \zeta_{r}^{l (a s)} y_{\overline{a s}}^{i}, \quad y_{\overline{a s}}^{i} := \sum_{b = 0}^{s - 1} \zeta_{s}^{i (b r)} x_{\overline{a s + b r}}. 
\end{align*}
From $\left\{ a s + b r \mid 0 \leq a \leq r - 1, \: 0 \leq b \leq s - 1 \right\} = \left\{ j - 1 + k r \mid 1 \leq j \leq r, \: 0 \leq k \leq s - 1 \right\}$, 
\begin{align*}
\sum_{a = 0}^{r - 1} \zeta_{r}^{l (a s)} y_{\overline{a s}}^{i} 
&= \sum_{a = 0}^{r - 1} \sum_{b = 0}^{s - 1} \zeta_{r}^{l (a s + b r)} \zeta_{s}^{i (a s + b r)} x_{\overline{a s + b r}} \\ 
&= \sum_{j = 1}^{r} \sum_{k = 0}^{s - 1} \zeta_{r}^{l (j - 1 + k r)} \zeta_{s}^{i (j - 1 + k r)} x_{\overline{j - 1 + k r}} \\ 
&= \sum_{j = 1}^{r} \zeta_{r}^{l (j - 1)} \sum_{k = 0}^{s - 1} \zeta_{s}^{i (j - 1 + k r)} x_{\overline{j - 1 + k r}}. 
\end{align*}
Therefore, denoting $x_{\overline{i}}$ by $x_{i + 1}$, 
we obtain Laquer's theorem.

\begin{proof}[Proof of Theorem~$\ref{thm:2}$]
Let $K := ( \mathbb{Z} / 2 \mathbb{Z} )^{l}$ and let $x_{g} \in \mathbb{Z}$ for all $g \in G$. 
For any $\chi \in \widehat{K}$, 
since $\Ima{(\chi)} \in \{ \pm 1 \}$, 
$\alpha_{\chi} := \Theta(H) \left[ x_{h} \mapsto \sum_{k \in K} \chi(k) x_{h k} \right]$ is an integer group determinant and 
$$
\alpha_{\chi} \equiv \alpha_{\chi_{1}} \pmod{2}, 
$$
where $\chi_{1}$ is the trivial character of $K$. 
Suppose that $\Theta(G)$ is even. 
Then from Theorem~$\ref{thm:1}$ and the definition of $M$, 
$\alpha_{\chi}$ are divisible by $2^{M}$ for all $\chi \in \widehat{K}$. 
Therefore, $\Theta(G) = \prod_{\chi \in \widehat{K}} \alpha_{\chi}$ is divisible by $(2^{M})^{|K|} = (2^{M})^{2^{l}}$, 
where $|K|$ denotes the order of $K$. 
\end{proof}

\noindent
\thanks{{\bf Acknowledgments}}
We would like to thank everyone who provided suggestions for improving this paper.

\bibliography{reference}

\begin{thebibliography}{10}

\bibitem{MR3879399}
Ton Boerkoel and Christopher Pinner.
\newblock Minimal group determinants and the {L}ind-{L}ehmer problem for
  dihedral groups.
\newblock {\em Acta Arith.}, 186(4):377--395, 2018.

\bibitem{MR1659232}
Keith Conrad.
\newblock The origin of representation theory.
\newblock {\em Enseign. Math. (2)}, 44(3-4):361--392, 1998.

\bibitem{Frobenius1968gruppen}
Ferdinand~Georg Frobenius.
\newblock \"{U}ber die {P}rimfactoren der {G}ruppendeterminante.
\newblock {\em Sitzungsberichte der K\"{o}niglich Preu{\ss}ischen Akademie der
  Wissenschaften zu Berlin}, pages 1343--1382, 1896.
\newblock Reprinted in {\it Gesammelte Abhandlungen, Band III}. Springer-Verlag
  Berlin Heidelberg, New York, 1968, pages 38--77.

\bibitem{Frobenius1968gruppencharaktere}
Ferdinand~Georg Frobenius.
\newblock \"{U}ber {G}ruppencharaktere.
\newblock {\em Sitzungsberichte der K\"{o}niglich Preu{\ss}ischen Akademie der
  Wissenschaften zu Berlin}, pages 985--1021, 1896.
\newblock Reprinted in {\it Gesammelte Abhandlungen, Band III}. Springer-Verlag
  Berlin Heidelberg, New York, 1968, pages 1--37.

\bibitem{MR1554141}
Thomas Hawkins.
\newblock The origins of the theory of group characters.
\newblock {\em Arch. History Exact Sci.}, 7(2):142--170, 1971.

\bibitem{MR1085397}
K.~W. Johnson.
\newblock On the group determinant.
\newblock {\em Math. Proc. Cambridge Philos. Soc.}, 109(2):299--311, 1991.

\bibitem{johnson2019group}
K.W. Johnson.
\newblock {\em Group matrices, group determinants and representation theory:
  The mathematical legacy of Frobenius}.
\newblock Lecture Notes in Mathematics. Springer International Publishing,
  2019.

\bibitem{MR2914452}
Norbert Kaiblinger.
\newblock Progress on {O}lga {T}aussky-{T}odd's circulant problem.
\newblock {\em Ramanujan J.}, 28(1):45--60, 2012.

\bibitem{MR1606416}
T.~Y. Lam.
\newblock Representations of finite groups: a hundred years. {I}.
\newblock {\em Notices Amer. Math. Soc.}, 45(3):361--372, 1998.

\bibitem{MR624127}
H.~Turner Laquer.
\newblock Values of circulants with integer entries.
\newblock In {\em A collection of manuscripts related to the {F}ibonacci
  sequence}, pages 212--217. Fibonacci Assoc., Santa Clara, Calif., 1980.

\bibitem{MR550657}
Morris Newman.
\newblock On a problem suggested by {O}lga {T}aussky-{T}odd.
\newblock {\em Illinois J. Math.}, 24(1):156--158, 1980.

\bibitem{MR4056860}
Christopher Pinner and Christopher Smyth.
\newblock Integer group determinants for small groups.
\newblock {\em Ramanujan J.}, 51(2):421--453, 2020.

\bibitem{MR803326}
B.~L. van~der Waerden.
\newblock {\em A history of algebra}.
\newblock Springer-Verlag, Berlin, 1985.
\newblock From al-Khw\={a}rizm\={\i} to Emmy Noether.

\bibitem{MR3622295}
Naoya Yamaguchi.
\newblock An extension and a generalization of {D}edekind's theorem.
\newblock {\em Int. J. Group Theory}, 6(3):5--11, 2017.

\end{thebibliography}
\bibliographystyle{plain}

\medskip
\begin{flushleft}
Naoya Yamaguchi\\
Faculty of Education \\ 
University of Miyazaki \\
1-1 Gakuen Kibanadai-nishi\\ 
Miyazaki 889-2192 \\
JAPAN\\
n-yamaguchi@cc.miyazaki-u.ac.jp
\end{flushleft}

\medskip
\begin{flushleft}
Yuka Yamaguchi\\
Faculty of Education \\ 
University of Miyazaki \\
1-1 Gakuen Kibanadai-nishi\\ 
Miyazaki 889-2192 \\
JAPAN\\
y-yamaguchi@cc.miyazaki-u.ac.jp
\end{flushleft}

\end{document}